\documentclass[12pt,reqno,twoside,letterpaper]{amsart}

\usepackage{hyperref}

\makeatletter
\renewenvironment{proof}[1][\proofname]{\par
  \pushQED{\qed}
  \normalfont \partopsep=\z@skip \topsep=\z@skip
  \trivlist
  \item[\hskip\labelsep
    \itshape
  #1\@addpunct{.}]\ignorespaces
}{
  \popQED\endtrivlist\@endpefalse\vspace{0.3em}
}
\makeatother

\newtheorem{prop}{Proposition}

\newcommand{\N}{\mathbb{N}}
\newcommand{\Pali}{\mathbb{P}}
\newcommand{\withbox}[2]{\mbox{\rlap{#1}}\hphantom{#2}}

\renewcommand{\epsilon}{\varepsilon}
\renewcommand{\phi}{\varphi}
\renewcommand{\rho}{\varrho}

\begin{document}

\title[On a conjecture of John Hoffman regarding sums of palindromic numbers]{On a conjecture of John Hoffman regarding \break sums of palindromic numbers}
\author{Markus Sigg}
\address{Freiburg, Germany}
\email{mail@markussigg.de}
\date{\today}

\begin{abstract}
  We disprove the conjecture that every sufficiently large natural number $n$
  is the sum of three palindromic natural numbers where one of them can be
  chosen to be the largest or second largest palindromic natural number smaller
  than or equal to $n$.
\end{abstract}

\maketitle

{
  AMS subjects classification 2010: 11B13.
}

\section{Introduction}

In the following, the terms \emph{digit} and \emph{palindromic} refer to
decimal representations. For $n \in \N$, its unique decimal representation is
given by
\begin{equation*}
  n = \sum_{j=0}^{h(n)} n_j \cdot 10^j.
\end{equation*}
with minimal $h(n) \in \N$ and digits $n_0,\dots,n_{h(n)} \in
\{0,\dots,9\}$. We identify $n$ with the digit string $n_{h(n)} \dots n_0$.

A natural number $n$ is called \emph{palindromic} iff $n_j = n_{h(n)-j}$ for $0
\le j \le n(h)$.

By $\Pali$ we denote the set of palindromic natural numbers, i.\,e.
\begin{equation*}
  \Pali = \{0,1,2,3,4,5,6,7,8,9,11,22,33,\dots,99,101,111,121,\dots\}.
\end{equation*}

Until recently, it was not known whether $\Pali$ is an additive basis of $\N$,
i.\,e.\ whether there exists $d \in \N$ such that $\N = d\,\Pali$, where
$d\,\Pali$ denotes the set of sums of $d$ elements of $\Pali$. William~D.~Banks
has in \cite{Banks} given a proof for $\N = 49\Pali$, which leaves still quite
some distance from the commonly conjectured $\N = 3\Pali$. \cite{Friedman}
mentions an even stronger conjecture of John~Hofmann, claiming that every
sufficiently large natural number $n$ is the sum of three elements of $\Pali$
where one of them can be chosen to be the largest or second largest palindromic
natural number $p \le n$. With the palindromic precursor and palindromic
successor
\begin{equation*}
  n_* := \max_{\Pali \ni p < n} p
  \quad\text{and}\quad
  n^* := \min_{\Pali \ni p > n} p\,,
\end{equation*}
and $n_{**} := \left(n_*\right)_*$ for $n \in \N$, the question is:

\begin{center}
  \emph{Is it true that $\{n - n_*, n - n_{**}\} \cap 2\Pali \ne \emptyset$ for
    every sufficiently large $n \in \N \setminus \Pali$?}
\end{center}

We are going to show that the answer is ``no''.

 \section{The Counterexample}

 The counterexample is constructed using 'non-$2\Pali$ twins', the palindromic
 twins $10^a \pm 1$ for suitable $a \in \N$ and the fact that the distance
 between a palindromic number $p$ and its successor $p^*$ can be arbitrarily
 large. As 'non-$2\Pali$ twins' we use the numbers $11 \cdot 10^k + 1$ and $11
 \cdot 10^k + 3$ for even $k$.

\begin{prop}
  $11 \cdot 10^k + 1 \not\in 2\Pali$ for $2 \le k \in \N$.
\end{prop}

\begin{proof}
  For $t := 11 \cdot 10^k + 1$ we have $h(t) = k+1$. Suppose $t = p+q$ with
  $p,q \in \Pali$ and $p \le q$, so $h(p) \le h(q) \le k+1$. Because $t \not\in
  \Pali$, we have $p > 0$.

  \withbox{(a)}{(a)} Suppose $h(q) = k+1$. Then $q_{k+1} = 1$, so $q_0 = 1$,
  so $p_0 = 0$, which is not possible.

  \withbox{(b)}{(a)} Suppose $h(p) = h(q) = k$. Because $t_0 = 1$ and $p_0,
  q_0 \neq 0$, we need $p_0 + q_0 = 11$, so $1$ is carried to the tens
  positions, and as this must add to $10$ with $p_1 + q_1$, we get $p_1 + q_1 =
  9$, and a $1$ is carried to the hundreds position. This goes on up to
  $p_{k-1} + q_{k-1} = 9$ and a carry to position $k$. But then $p+q \ge (p_k +
  q_k + 1) \cdot 10 ^k = (p_0 + q_0 + 1) \cdot 10^k = 12 \cdot 10^k > t$.

  \withbox{(c)}{(a)} Suppose $h(p) < h(q) = k$. Then $p+q \le (10^k-1) +
  (10^{k+1}-1) = 11 \cdot 10^k - 2 < t$.

  \withbox{(d)}{(a)} Suppose $h(p) \le h(q) < k$. Then $p+q \le (10^k-1) +
  (10^k-1) = 2 \cdot 10^k - 2 < t$.
\end{proof}

\begin{prop}
  $11 \cdot 10^k + 3 \not\in 2\Pali$ for $2 \le k \in \N$, $k$ even.
\end{prop}

\begin{proof}
  For $t := 11 \cdot 10^k + 3$ we have $h(t) = k+1$. Suppose $t = p+q$ with
  $p,q \in \Pali$ and $p \le q$, so $h(p) \le h(q) \le k+1$. Because $t \not\in
  \Pali$, we have $p > 0$.

  In the following, for a digit $\alpha$ and $m \in \N$, ${[\alpha]}_m$ denotes
  the concatenation of $m$ copies of $\alpha$.

  \withbox{(a)}{(a)} Suppose $h(q) = k+1$. Then $q_{k+1} = 1$, so $q_0 = 1$,
  so $p_0 = 2$, so $p_{h(p)} = 2$, so $h(p) < k$. A carry is needed from
  position $h(p)$ to position $h(p) + 1$ to get ${(p+q)}_{h(p)} = 0$, and so
  on, up to a carry from position $k-1$ to position $k$. With this carry, we
  would get $p + q > t$ if $q_k > 0$, so $q_k = 0$, so $q_1 = 0$. For $h(p) >
  1$ we get $p_1 = 0$. For $h(p) > 2$ we get $p_{h(p)-1} = 0$.

  \withbox{(aa)}{(aa)} Suppose $k = 2$. Then $q = 1001$ and $p \in \{2,22\}$,
  so $p+q \neq t$.

  \withbox{(ab)}{(aa)} Suppose $k = 4$. Then $q = 10\delta\delta01$ with a
  digit $\delta$ and $p \in \{2,22,202,2002\}$, so $p+q \neq t$.

  \withbox{(ac)}{(aa)} Suppose $k \ge 6 \wedge h(p) \le 5$. Then $q =
  10\delta\epsilon\alpha\epsilon\delta01$ with digits $\delta$ and $\epsilon$
  and a palindromic digit string $\alpha$ which is empty in case of $k = 6$.
  To get ${(p+q)}_k = 1$, $\delta = 9$ is needed, so $q =
  109\epsilon\alpha\epsilon901$ and $p+q \neq t$ for $p \in
  \{2,22,202,2002\}$. For $p = 20\phi02$ with some digit $\phi$, to get
  ${(p+q)}_2 = 0$ we need $\phi = 1$ and $\epsilon = 9$, but then in case of $k
  = 6$ we get $p+q = 20102 + 10999901 = 11020003 \ne t$, while in case of $k >
  6$ we need $\alpha = 7{[9]}_{k-8}7$, so $p+q = 20102 + 10997{[9]}_{k-8}79901
  = 10998{[0]}_{k-8}00003 \neq t$. For $p = 20\phi\phi02$ with some digit
  $\phi$, to get ${(p+q)}_2 = 0$ we need $\phi = 1$ and $\epsilon = 8$, but
  then in case of $k = 6$ we get $p+q = 201102 + 10988901 = 11190003 \neq t$,
  while in case of $k > 6$ we have
  \begin{equation*}
    p+q < 10^6 + 1099 \cdot 10^{k-2} \le 10^{k-2} + 1099 \cdot 10^{k-2} = 11
    \cdot 10^k < t.
  \end{equation*}
  \withbox{(ad)}{(aa)} Suppose $k \ge 8 \wedge h(p) \ge 6$. Then $q =
  10\delta\epsilon\alpha\epsilon\delta01$ and $p = 20\phi\beta\phi02$ with
  digits $\delta,\epsilon,\phi$ and non-empty palindromic digit strings
  $\alpha,\beta$. We will construct $p',q' \in \Pali, p' \le q'$ with $h(q') =
  k-1$ and $p'+q' = 11 \cdot 10^{k-2} + 3$, which gives rise to an impossible
  infinite descent.

  \withbox{(ada)}{(ada)} Suppose $\phi = 0$. Then $\delta = 0$, hence $q =
  100\epsilon\alpha\epsilon001$ and $p = 200\beta002$, and we can take $q' :=
  10\epsilon\alpha\epsilon01$ and $p' := 20\beta02$.

  \withbox{(adb)}{(ada)} Suppose $\phi \neq 0$ and $h(p) = k-1$. We have $\phi
  + \delta = 10$ and $\delta \neq 0$, and $\beta$ must have at least two
  digits, i.\,e.\ $\beta = \psi\gamma\psi$ with a digit $\psi$ and a (possibly
  empty) palindromic digit string $\gamma$, so $p =
  20\phi\psi\gamma\psi\phi02$, which allows to take $q' :=
  10\delta\alpha\delta01$ and $p' := 20\phi\gamma\phi02$.

  \withbox{(adc)}{(ada)} Suppose $\phi \neq 0$ and $h(p) < k-1$. We have $\phi
  + \delta = 10$, and $h(p) < k-1$ leads to $\delta = 9$ and $\phi = 1$, so $q
  = 109\epsilon\alpha\epsilon901$ and $p = 201\beta102$.

  \withbox{(adca)}{(adca)} Suppose $\beta$ is more than one digit,
  i.\,e.\ $\beta = \psi\gamma\psi$ with a digit $\psi$ and a (possibly empty)
  palindromic digit string $\gamma$, hence $p = 201\psi\gamma\psi102$. Then we
  take $q' := 109\alpha901$ and $p' := 201\gamma102$.

  \withbox{(adcb)}{(adca)} Suppose $\beta$ is a single digit. As $k$ is even,
  $\alpha$ has an even number of digits. If $\alpha$ were two digits, say
  $\alpha = \tau\tau$ with a digit $\tau$, so $q =
  109\epsilon\tau\tau\epsilon901$, we would need $\tau = 8$ for the lower
  position, but $\tau = 9$ for the higher position of $\tau$. If $\alpha$ were
  more than two digits, say $\alpha = \tau\rho\tau$ with a digit $\tau$ and a
  palindromic digit string $\rho$ with $2$ or more digits, so $q =
  109\epsilon\tau\rho\tau\epsilon901$, we would again need $\tau = 8$ for the
  lower position, but $\tau = 9$ for the higher position of $\tau$. So the case
  (adcb) is not possible at all.

  \withbox{(b)}{(a)} Suppose $h(p) = h(q) = k$. Then $p_0 + q_0 = 3, p_k + q_k
  \in \{10,11\}$, but $p_k = p_0, q_k = q_0$.

  \withbox{(c)}{(a)} Suppose $h(p) < h(q) = k$. Then $p+q \le (10^k-1) +
  (10^{k+1}-1) = 11 \cdot 10^k - 2 < t$.

  \withbox{(d)}{(a)} Suppose $h(p) \le h(q) < k$. Then $p+q \le (10^k-1) +
  (10^k-1) = 2 \cdot 10^k - 2 < t$.
\end{proof}

\begin{prop}
  There are infinitely many $n \in \N \setminus \Pali$ with $n-n_*,n-n_{**}
  \not\in 2\Pali$.
\end{prop}

\begin{proof}
  Let $1 \le j \in \N$. Then for $t := 11 \cdot 10^{2j} + 1$, propositions 1
  and 2 show $t,t+2 \not\in 2\Pali$. Take $m \in \N$ with $10^m > t$ and set $p
  := 10^{2m} + 1 \in \Pali$. Then $p^* = 10^{2m} + 10^m + 1 = p + 10^m$ and
  $p_* = 10^{2m} - 1 = p-2$. For $n := p+t$ we have $p < n < p+10^m = p^*$, so
  $n \not\in \Pali$ and $n_* = p$, hence $n - n_* = n - p = t \not\in 2\Pali$
  and $n - n_{**} = n - p_* = n - (p-2) = t+2 \not\in 2\Pali$.

  In this way, for every $j \ge 1$ choose an $m(j)$ and get an $n(j)$ with the
  desired properties. Taking $m(j+1) > m(j)$ gives $n(j+1) > n(j)$.
\end{proof}

Choosing the smallest possible $m$ with $10^m > 11 \cdot 10^{2j} + 1$, namely
$m = 2j + 2$, in the proof of proposition 3 yields $n(j) = 10\,000^{j+1} + 11
\cdot 100^j + 2$.

On a related note, we would like to point out that the greedy algorithm which,
given a natural number, repeatedly subtracts the largest possible palindromic
number, can result in an arbitrarily large number of palindromic summands:
Start with $n(1) := 1$. To get $n(j+1)$, take $m \in \N$ with $10^{m} > n(j)$
and set $n(j+1) := 10^{2m} + 1 + n(j)$. Then $n(j+1) \not\in \Pali$ and
${n(j+1)}_* = 10^{2m} + 1$, so $n(j+1) - {n(j+1)}_* = n(j)$. For every $j \in
\N$, the greedy algorithm partitions $n(j)$ into $j$ palindromic summands.
Consequently, and in confirmation of a recent presumption of Neil~Sloane
\cite{OEIS}, the OEIS sequence A088601 is unbounded.

\end{document}